\documentclass[reqno]{amsart}
\usepackage{amsmath,amsfonts,euscript,amscd,amsthm,amssymb,upref,graphics,color,verbatim}
\usepackage[normalem]{ulem}
\usepackage{enumitem}
\usepackage{graphics}
\usepackage{epsfig}

\theoremstyle{plain}
\newtheorem{theorem}{Theorem}

\newtheorem{lemma}[theorem]{Lemma}
\newtheorem{corollary}[theorem]{Corollary}

\theoremstyle{definition}

\newtheorem{remark}[subsection]{Remark}

\newtheorem{nothing*}[subsection]{}

\theoremstyle{remark}
\newcommand{\rien}[1]{}

\renewcommand{\epsilon}{\varepsilon}
\renewcommand{\phi}{\varphi}
\renewcommand{\emptyset}{\varnothing}

\title{On the geometry of  simply connected wandering domains}

\author{Luka Boc Thaler}

\begin{document}
\thanks{The author was supported by the research program P1-0291 from ARRS, Republic of Slovenia}
\address{Faculty of Education, University of Ljubljana, Kardeljeva ploščad 16, 1000 Ljubljana, Slovenia}
\address{ Institute of Mathematics, Physics, and Mechanics, Jandranska 19, 1000 Ljubljana, Slovenia.} \email{luka.boc@pef.uni-lj.si}

\begin{abstract}
 We study the geometry of simply connected wandering domains for entire functions and we prove that every bounded connected regular open set, whose closure has a connected complement, is a wandering domain of some entire function. In particular such a domain can be realized as an escaping or an oscillating wandering domain. As a consequence we obtain that every Jordan curve is the boundary of a wandering Fatou component of some entire function.
\end{abstract}
%
%
%
%
%
\maketitle

\section{Introduction}
A general goal in discrete dynamical systems is to describe qualitatively the possible
dynamical behaviour under iteration of maps satisfying certain conditions, be they
algebraic or analytic. In this paper we consider the dynamical system given by the iterates of an entire function $f:\mathbb{C}\rightarrow \mathbb{C}$. We use $f^n$ to denote the $n$'th iterate of $f$.  There is a natural dichotomy of the complex plane associated to such dynamical system. We say that a point $p\in\mathbb{C}$ belongs to the \emph{Fatou set} $\mathcal{F}_f$ if and only if there exists an open neighbourhood of $p$ on which the sequence of iterates $(f^n)$ forms a normal family. The Julia set $\mathcal{J}_f$ is then defined as the complement of the Fatou set. The connected components of the Fatou set are called the \emph{Fatou components}. We say that a Fatou component $\Omega$ is \emph{pre-periodic} if there are non-negative integers $n\neq m$ such that $f^n(\Omega)\cap f^m(\Omega)\neq\emptyset$. A Fatou component which is not pre-periodic is called a wandering Fatou component or a \emph{wandering domain}.  

 One of the main goals in complex dynamics is to obtain a complete classification of all possible Fatou components for a given class of maps in terms of their dynamics and their geometry.
 
In the class of entire functions we have a complete classification of the pre-periodic Fatou components, see \cite{Mil06, Sch10}, and the recent developments in studies of wandering domains show that we are also getting closer to a complete classification of wandering domains.

By Sullivan's non-wandering theorem \cite{Sul85} we know that polynomials cannot have wandering domains, hence they can only appear  in the class of the entire transcendental functions. 
 The first example of such a domain  was given by Baker \cite{Bar76} who proved that a certain entire function has a multiply connected Fatou component, which is necessarily a wandering domain \cite{Bar75}.  Since then several examples of simply connected wandering domains have been constructed, see \cite{Her84, Bar84, EL87}. In \cite{KS} Kisaka-Shishikura proved that the eventual connectivity of a wandering domain is constant, equal to either $1$, $2$ or $\infty$. Moreover they showed that for every $p \geq  2$ there exists an entire function with a Fatou component of connectivity $p$.

On a wandering  domain  all  limit  functions  must  be  constant  \cite{Fat20}, therefore we can classify wandering domains into the following three classes. Those  wandering domains for which the only limit function is the point at infinity are called \emph{escaping}, while the rest are either \emph{oscillating} (if the point at infinity and some finite point are both limit functions) or \emph{dynamically bounded} (if all limit functions are finite points).  
%

Note that Baker's first wandering domain was of the escaping type. The first example of an oscillating wandering domain was given by Eremenko and Lyubich in \cite{EL87}. 
In the same paper they proved the existence of an entire function $f$ having a wandering domain $\Omega$, so that $f^n|_{\Omega}$ is univalent for all $n\geq 1$. They were also the first who used an approximation theory for the construction of entire functions in complex dynamics. A major open problem in transcendental dynamics is whether dynamically bounded wandering domains exist at all.

A complete description of the interior dynamical behaviour in a multiply connected wandering domain was given in \cite{BRS13}. In \cite{BEFRS19} the authors classified simply connected wandering domains in terms of the hyperbolic distance between orbits of points and in terms of convergence to the boundary. They show that there are in total nine types of  simply connected wandering domains all of which can be realized as an escaping wandering domain. Recently, in \cite{ERS20} the authors have shown that only six out of these nine types can be realized by an oscillating wandering domain. Let us just mention that there are several important works proving the existence of wandering domains for entire functions in the Eremenko-Lyubich class that consists of all entire functions with bounded singular set, see \cite{Bis15,FJL19, MPS20}. Whether or not all of the previously discussed types of wandering domains can be realized in this class is presently unknown.

Throughout this paper we will use $\Delta(p,r)$ to denote an open disk of radius $r$ centered at $p$. We will also use $\overline{U}$, to denote the closure of the set $U$, and $\partial U$ to denote the boundary of $U$. By $V\Subset U$ we mean that $V$ is compactly contained in $U$.

The aim of this paper is to study the geometry of simply connected wandering domains of entire functions. In particular we study which bounded simply connected domains can be realized as a wandering domain of an entire function. Recall that an open set $U$ is called \emph{regular} if it coincides with the interior of its closure. The following are our main results.

\begin{theorem}\label{thm:esc} Let $\Omega\subset\mathbb{C}$ be a bounded connected regular open set whose closure has a connected complement. There exists an entire function $f$ for which $\Omega$ is an escaping wandering domain and the iterates $f^n|_\Omega$ are univalent.
\end{theorem}

\begin{theorem}\label{thm:osc} Let $\Omega\subset\mathbb{C}$ be a  bounded connected regular open set whose closure has a connected complement. There exists an entire function $f$ for which $\Omega$ is an oscillating wandering domain and the iterates $f^n|_\Omega$ are univalent.  
\end{theorem}

 As an immediate consequence of these results we get the following corollary.

\begin{corollary} Every Jordan curve is the boundary of a wandering Fatou component of some entire function.
\end{corollary}


Note that every regular open set whose closure has a connected complement is a simply connected set. The condition that $\Omega$ is a regular open set is not only sufficient but also necessary. Indeed, let $\Omega$ be a wandering domain of an entire function $f$ and assume that $\Omega$ is not regular. Since $\text{int}(\overline{\Omega})\cap\mathcal{J}_f\neq \emptyset$ it follows that $\cup_{n=0}^{\infty}f^n(\text{int}(\overline{\Omega}))$ covers the whole plane with at most one exception. On the other hand since $\Omega$ is a wandering domain and $f$ is a continuous open map we have $f^n(\text{int}(\overline{\Omega}))\cap f^m(\text{int}(\overline{\Omega}))=\emptyset$ for all $n\neq m$ which brings us to the contradiction.  
%
%
%


 The other two conditions in our theorems, namely that $\Omega$ is bounded and that  $\mathbb{C}\backslash \overline{\Omega}$ is connected, are needed for the application of the following stronger version of the well-known Runge’s Approximation Theorem.  

\begin{theorem}\label{approx} Let $K_1,\ldots,K_n \subset\mathbb{C}$ be pairwise disjoint compact sets whose complements $\mathbb{C}\backslash K_j $ are connected. Let $L_k\subset K_k$ be a finite set of points and $f_k:K_k\rightarrow \mathbb{C}$ a holomorphic map  for every $1\leq k\leq n$.  For every $\epsilon>0$ there exists an entire function $f$ satisfying:
\begin{enumerate}
\item $\|f_k-f\|_{K_k}<\epsilon$
\item  $f(x)=f_k(x)$ for all $x\in L_k$
\item $f'(x)=f_k'(x)$  for all $x\in L_k$
\end{enumerate}
for every $1\leq k\leq n$.
\end{theorem}

For the proof of Theorem \ref{approx} we refer the reader to the Appendix. Let us remark that we do not know whether the condition of $\mathbb{C}\backslash\overline{\Omega}$ being connected is also a necessary condition. In fact this problem is closely related to the question about the existence of shielded components (also known as hidden components), see \cite{BF,CR}. 

{\bf Question:} \emph{Is it true that the closure of any bounded simply connected Fatou component, of an entire function, has a connected complement?}

If the answer is positive, then our result describes all possible geometries of bounded simply connected wandering domains.


\medskip

Finally we give a brief sketch of the proof of our theorems. Observe that $\Omega$ is a regular open set if and only if $\partial\Omega=\partial(\mathbb{C}\backslash\overline{\Omega})$, hence $\Omega$ is regular if and only if there exists a sequence of points $(x_n)\in \mathbb{C}\backslash\overline{\Omega}$ which accumulates everywhere on $\partial\Omega$. Using Theorem \ref{approx} we inductively construct a sequence of entire functions that converges uniformly to the entire function $f$ with the following properties. The $f$-iterates of $\Omega$ have uniformly bounded diameters, which implies that $\Omega$ is contained in the Fatou set.  We make sure that $f^n|_{\Omega}$ is univalent for all $n\geq 1$, and that the points $x_n$ are pre-images of an attracting fixed point of $f$. Since these points accumulate everywhere on $\partial\Omega$ it follows that $\Omega$ is a Fatou component. The fact that the orbit of $\Omega$ is escaping/oscillating will follow directly from the construction.

Similar ideas were previously used by the author \cite{BT21}, to prove the existence of bounded wandering domains for holomorphic automorphisms of $\mathbb{C}^n$. In particular the author proved that any bounded regular open set, whose closure is polynomially convex, is a wandering domain of some holomorphic automorphism of $\mathbb{C}^n$. For example any convex domain in $\mathbb{C}^n$ satisfies these conditions. 
 
 \subsection*{Acknowledgments} I would like to express my sincere gratitude to the anonymous referees for their constructive remarks which lead to a much improved presentation. In particular I would like to thank the referee for pointing out that my hypothesis, that $\Omega$ is regular, is not only sufficient, but also necessary.

\section{Proof of Theorem 1}
We start the proof by choosing a sequence $(U_n)_{n\geq0}$ of compact neighbourhoods of $\overline{\Omega}$ that  satisfies:  
\begin{enumerate}[label=(\roman*)]
\item $\mathbb{C}\backslash U_n$ is connected for all $n \geq 0$,
\item  $U_{n+1}\subset \text{int}(U_n)$ for all $n \geq 0$,
\item $\overline{\Omega} =\bigcap_{n\geq 0}U_n$.
\end{enumerate}
Note that such a sequence can be obtained by defining $U_n$ as the union of $V_n=\{z\in\mathbb{C}\mid d(z,\overline{\Omega})\leq \frac{1}{n}\}$ and all bounded components of $\mathbb{C}\backslash V_n$.
Next we choose a sequence $(x_n)\in\mathbb{C}\backslash \overline{\Omega}$ that satisfies the following properties:
\begin{enumerate}
\item $x_n\in \text{int}(U_{n-1})\backslash U_n$  for all $n \geq 1$,
\item $\omega((x_n))=\partial\Omega$, i.e. $(x_n)$ accumulates everywhere on the boundary of $\Omega$.
\end{enumerate}

We assume, without loss of generality, that $3\in \Omega \subset U_0\subset \Delta(3,1)$. Note that this can be always achieved with a linear change of coordinates.

 As we have mentioned in the introduction the idea is to construct an entire function $f$ so that the iterates of $\Omega$ will escape every compact set and will have a uniformly bounded diameter. Moreover we will make sure that all iterates of $f$ are univalent on $\Omega$ and that the points $x_n$ are pre-images of an attracting fixed point of $f$. Such function will be obtained as a limit of an inductively constructed sequence of entire functions given by the following lemma.

\begin{lemma}\label{propescape} Let  $(U_n)_{n\geq 0}$ and $(x_n)_{n\geq 1}$ be as above. There exists a sequence $(f_k)_{k\geq 1}$ of entire functions and sequences of points  $(x_n^j)_{n >  j \geq 1 }$,  such that the following properties are satisfied for all $k\geq 1$:

\begin{enumerate}[label=(\alph*)]
\item   $\|f_{k+1}-f_{k}\|_{\overline{\Delta}(0,4k+1)}\leq 2^{-k}$,
\item $f_{k}^n(U_k)\subset \Delta(4n+3,1),$ for all $1\leq n\leq k$,
\item $f_{k}^n|_{U_k}$ is univalent for all $1\leq n\leq k$,
\item $f_k^j(x_n)=x^{j}_n$ for all  $1\leq j< n \leq k$,
\item $f_k^n(x_n)=0$ for all $1\leq n\leq k$,
\item $f_k(0)=0$ and $f'_k(0)=\frac{1}{2}$.
\end{enumerate}
\end{lemma}


\begin{proof}[Proof of the Lemma:] 
Let us first construct a function $f_1$ that satisfies:
\begin{enumerate}
\item  $\|f_{1}(z)-\frac{1}{2}z\|_{\overline{\Delta}(0,1)}\leq 2^{-1}$,
\item $f_{1}(U_1)\subset \Delta(7,1),$ 
\item $f_{1}|_{U_1}$ is univalent,
\item $f_1(x_1)=0$,
\item $f_1(0)=0$ and $f'_1(0)=\frac{1}{2}$.
\end{enumerate}
Let $K_1=\overline{\Delta}(0,1)$, $K_2=\{x_1\}$ and $K_3=U_1$ be disjoint compact sets and observe that their  complement is connected. Moreover let $h_1(z)=\frac{1}{2}z$, $h_2(z)=0$ and $h_3(z)=z+4$. By Theorem \ref{approx} for every $\epsilon_1>0$ there exists an entire function $f_1$ such that 
\begin{itemize}
\item  $\|f_{1}-h_j\|_{K_j}\leq \epsilon_1$ for all $1\leq j\leq 3$
\item $f_1(x_1)=0$,
\item $f_1(0)=0$ and $f'_1(0)=\frac{1}{2}$.
\end{itemize}
By choosing $\epsilon_1$ sufficiently small we also get $f_{1}(U_1)\subset \Delta(7,1)$ and $f_{1}|_{U_1}$  univalent. Finally we define $x_2^1:=f_1(x_2)$ and observe that $f_1$ satisfies all the conditions of the lemma for $k=1$.
\medskip

Now let us assume that we have already constructed entire functions $f_1,\ldots ,f_k$ and points $x_n^j$, where $1\leq j<n\leq k+1$,  that satisfy all conditions $(a)-(f)$. We proceed with the constructions satisfying the conditions for $k + 1$.
\medskip

By the induction hypothesis we know that $f^k_k|_{U_{k+1}}$ is univalent and $f^k_k(U_{k+1})\subset  \Delta(4k+3,1)$. Moreover  $x_{k+1}^{k}\notin f^k_k(U_{k+1})$ and  $\mathbb{C}\backslash f^k_k(U_{k+1})$ is connected. Let us define compact sets
$$
K_1=\overline{\Delta}(0,4k+1) ,\quad K_2=\{x_{k+1}^{k}\}
$$
and choose a compact set $K_3$ that satisfies 
\begin{itemize}
\item $f^k_k(U_{k+1})\subset \text{int}(K_3)$  
\item $K_3\cap(K_1\cup K_2)=\emptyset$,
\item $\mathbb{C}\backslash K_3$ is connected.
\end{itemize}
Next we define functions
 $$
 h_1(z)=f_k(z),\quad  h_2(z)= 0, \quad h_3(z)=z+4.
 $$

Clearly $K_1$, $K_2$ and $K_3$ satisfy all the conditions of Theorem \ref{approx}, therefore for every $\epsilon_{k+1}>0$ there exists an entire function $f_{k+1}$, such that: 
\begin{enumerate}
\item $\|f_{k+1}-h_j\|_{K_j}\leq \epsilon_{k+1}$ for every $1\leq j\leq 3$, 
\item $f_{k+1}(x_n^j)=f_k(x_n^j)$ for all $0 \leq j < n \leq k$, (here $x_n^0:=x_n$)
\item $f_{k+1}(x_{k+1}^j)=f_k(x_{k+1}^j)$ for all $0 \leq j < k$, (here $x_{k+1}^0:=x_{k+1}$)
\item $f_{k+1}(x_{k+1}^{k})=0$,
\item $f_{k+1}(0)=0$ and $f_{k+1}'(0)=\frac{1}{2}$,

\end{enumerate}

where we can choose  $\epsilon_{k+1}\leq \frac{1}{2^k}$ small enough such that:
\begin{enumerate}[label=(\roman*)]
\item $f_{k+1}^n(U_{k+1})\subset  \Delta(4n+3,1)$, for all $1\leq n\leq k+1$,
\item  $f_{k+1}^n|_{U_{k+1}}$ is univalent for all $1\leq n\leq k+1$.
\end{enumerate}
 Indeed, for (i) recall that by the induction hypothesis we have $f_{k}^n(U_{k})\subset  \Delta(4n+3,1)$, for all $1\leq n\leq k$, hence for all small $\epsilon_{k+1}$ we also have  $f_{k+1}^n(U_{k+1})\subset  \Delta(4n+3,1)$, for all $1\leq n\leq k$.  By taking even smaller $\epsilon_{k+1}$ if necessary, we may assume that $f_{k+1}^k(U_{k+1})\subset K_3$, hence $f_{k+1}^{k+1}(U_{k+1})\subset  \Delta(4(k+1)+3,1)$ follows from the fact that $h_3(z)=z+4$ on $K_3$.
 
  For (ii)  recall that by the induction hypothesis $f_k^n|_{U_{k}}$ is univalent for all $1\leq n\leq k$, hence $f_k$ is univalent on $\cup^{k-1}_{n=0}f_k^n(U_{k})\subset K_1$. Since we $U_{k+1}\subset \text{int}(U_{k})$ we can choose $\epsilon_{k+1}$ so small that $f_{k+1}^n|_{U_{k+1}}$ is univalent for all $1\leq n\leq k$. By taking even smaller $\epsilon_{k+1}$ if necessary, we may assume that $f_{k+1}^k(U_{k+1})\subset K_3$ and since $h_3$ is univalent on $K_3$ we can conclude that also $f_{k+1}^{k+1}|_{U_{k+1}}$ is univalent.
\medskip

Finally we define $x_{k+2}^j:=f_{k+1}^j(x_{k+2})$ for all $1\leq j\leq  k+1$. The entire function $f_{k+1}$ now satisfies all the conditions $(a)-(f)$ of the Lemma, hence this completes the inductive step.

\end{proof}

Let us continue with the proof of Theorem 1.  Let $(f_n)$ be a sequence of entire functions as given by the lemma above. The sequence of entire functions $(f_n)$  converges uniformly on compacts to an entire function $f$ satisfying:
\begin{enumerate}
\item $f^n(\Omega)\subset \Delta(4n+3,1),$ for all $n\geq 0$,
\item $f^n|_{\Omega}$ is univalent for all $n\geq 0$
\item $f^n(x_n)=0$ as for all $n\geq 1$,
\item $f(0)=0$ and $f'(0)=\frac{1}{2}$. 

\end{enumerate}

\begin{figure}[t]
\label{fig:escape}
\includegraphics[width=6in]{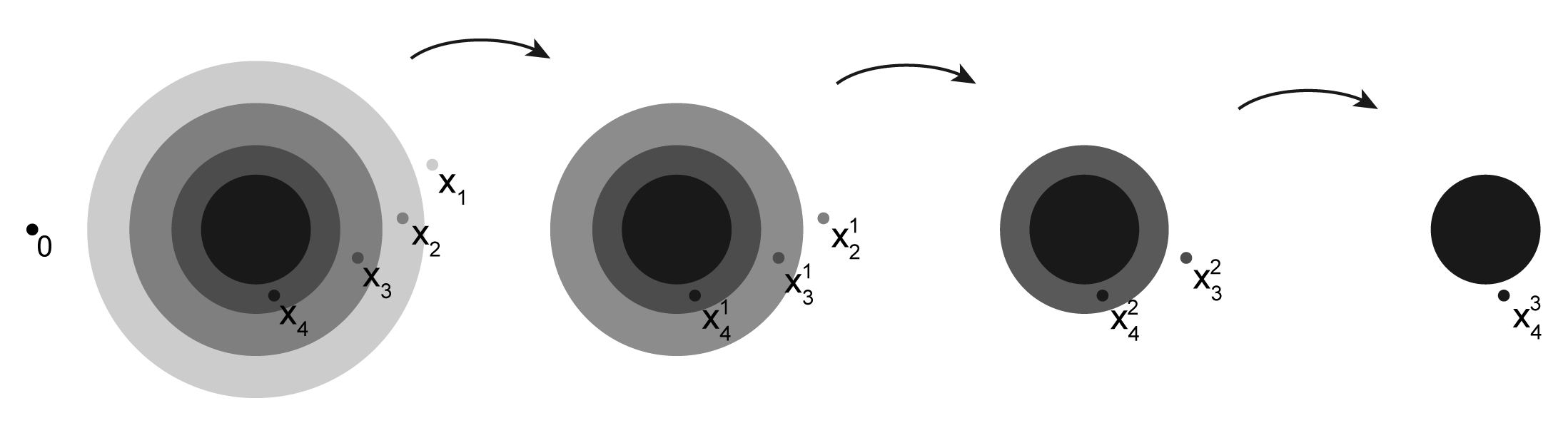}
\caption{The figure shows the first three iterates of $f$ acting on the wandering domain $\Omega$, which lies inside the black disk. There are four points around this disk which are part of the sequence of points that accumulates everywhere on the boundary of $\Omega$. Points $x_1$, $x_{2}^1$, $x_{3}^2$ and  $x_{4}^3$, that belong to the orbit of a point $x_1$, $x_2$, $x_3$ and $x_4$ respectively,  are mapped to the attracting fixed point at the origin. The action of $f$ on the orbit of $\Omega$ is approximately a translation by $4$. } 
\end{figure}

It follows from $(1)$  that the orbit of $\Omega$ leaves every compact set and that the Euclidean diameter of $f^n(\Omega)$ is uniformly bounded for all $n\geq 0$, hence $\Omega$ is contained in the Fatou set.  By (4)  the origin is an attracting fixed point of $f$ whose basin of attraction is clearly disjoint from $\Omega$. Finally (3) implies that the pre-images of the origin accumulate everywhere on the boundary of $\Omega$, hence $\Omega$ is a Fatou component.

\begin{remark}\label{rem:1} According to the classification in \cite{BEFRS19} a simply connected wandering domain  $\Omega$ is \emph{eventually isometric} if and only if there are $z,z'\in \Omega$ for which the hyperbolic distance $\text{dist}_{\Omega_n}(f^n(z),f^n(z'))=c(z,z')>0$ for all sufficiently large $n$, where $\Omega_n$ is a Fatou component containing $f^n(\Omega)$ and  $c(z,z'):=\lim_{n\rightarrow \infty}\text{dist}_{\Omega_n}(f^n(z),f^n(z'))$. Note that if this is the case for the points $z,z'$ then the same holds of all pairs of points in $\Omega$ that are not pre-images of the same point. In our case the iterates of $f$ are univalent on $\Omega$ and $\Omega_n=f^n(\Omega)$, therefore $\text{dist}_{\Omega}(z,z')=\text{dist}_{\Omega_n}(f^n(z),f^n(z'))$ for all $n>0$. This shows that the wandering component in Theorem \ref{thm:esc} is eventually isometric.

 With a small adaptation of the proof we can also obtain any of the three possible types of behaviour of the orbits of points in terms of convergence to the boundary:
 \begin{itemize}
 \item   \emph{convergent} type, i.e. $\text{dist}_{\text{Eucl}}(f^n(z),\partial \Omega_n)\rightarrow 0$ for all $z\in \Omega$.
 \item  \emph{bungee} type, i.e. there are sequences $(n_k)$ and $(m_k)$ such that $\text{dist}_{\text{Eucl}}(f^{n_k}(z),\partial \Omega_{n_k})\rightarrow 0$ and 
 $\liminf_{k\rightarrow \infty}\text{dist}_{\text{Eucl}}(f^{m_k}(z),\partial \Omega_{m_k})>0$ for all $z\in \Omega$.
 \item all  orbits  \emph{stay away}  from  the boundary, i.e. $
\liminf_{n\rightarrow \infty}\text{dist}_{\text{Eucl}}(f^n(z),\partial \Omega_n)>0.
$ for all $z\in \Omega$.
\end{itemize}

The least work is needed to make sure that all orbits stay away from the boundary. Indeed, let $h(z)=z+4$ and observe that for every $n\geq0$ our function $f$ satisfies $\|h^n-f^n\|_{\overline\Omega}<\sum_{k=0}^n\sum_{j=k+1}^\infty\epsilon_j$. Now given a point $z'\in\Omega$ we can chose $\epsilon_j$'s sufficiently small so that   $\sum_{n=0}^\infty\sum_{j=n+1}^\infty\epsilon_j< \frac{1}{2}\text{dist}_{\text{Eucl}}(z',\partial\Omega)$.  This implies that
\begin{align*}\text{dist}_{\text{Eucl}}(f^n(z'),\partial \Omega_n)&\geq
\text{dist}_{\text{Eucl}}(h^n(z'),h^n(\partial\Omega)) -|f^n(z')-h^n(z')|-\sup_{w\in\partial \Omega}|f^n(w)-h^n(w)|\\
&\geq\text{dist}_{\text{Eucl}}(z',\partial\Omega)-2\|h^n-f^n\|_{\overline\Omega}>0,
\end{align*}
for all $n\geq0$. Since the above inequality holds for the point $z'\in\Omega$ it follows that it holds for all points in $\Omega$, see \cite[Theorem 4.2]{BEFRS19}, hence this concludes the argument. 

Finally let us mention that in order to achieve the convergent type and the bungee type one only needs to make an appropriate choice of a linear map $h_3$ in every step of the induction.
\end{remark}

\section{Proof of Theorem 2}

As before we start the proof by choosing a sequence $(U_n)_{n\geq0}$ of compact neighbourhoods of $\overline{\Omega}$ that  satisfies:  
\begin{enumerate}[label=(\roman*)]
\item $\mathbb{C}\backslash U_n$ is connected for all $n \geq 0$,
\item  $U_{n+1}\subset \text{int}(U_n)$ for all $n \geq 0$,
\item $\overline{\Omega} =\bigcap_{n\geq 0}U_n$.
\end{enumerate}

Next we choose a sequence $(x_n)\in\mathbb{C}\backslash \overline{\Omega}$ that satisfies the following properties:
\begin{enumerate}
\item $x_n\in \text{int}(U_{n-1})\backslash U_n$  for all $n \geq 1$,
\item $\omega((x_n))=\partial\Omega$, i.e. $(x_n)$ accumulates everywhere on the boundary of $\Omega$.
\end{enumerate}

We assume, without the loss of generality, that $\Omega\subset U_0\Subset \Delta(\frac{2}{3},\frac{1}{9})$. Note that this can be achieved with a linear change of coordinates.

%
%

We will obtain or function $f$ by taking a limit of an inductively constructed sequence of entire functions given by the following lemma.

%
%

\begin{lemma}\label{lem:2}  Let  $(U_n)_{n\geq 0}$ and $(x_n)_{n\geq 1}$ be as above. Let $A_k:=\Delta(0,\frac{1}{2k+1})\backslash \overline{\Delta}(0,\frac{1}{2k+3})$ and let $N_k:=\frac{k(k+1)}{2}$. There exists a sequence $(f_k)_{k\geq 1}$ of entire function and a sequence  $(x^j_n)_{n,j\geq 1 }$ of points such that the following properties are satisfied for all $k\geq 1$:

\begin{enumerate}[label=(\alph*)]
\item   $\|f_{k+1}-f_{k}\|_{\overline{\Delta}(0,4k-2)}\leq 2^{-k}$,
\item $f_{k}^{N_k}(U_k)\subset A_k$
\item $f_{k}^{n}(\overline{\Delta}(0,\frac{1}{2k+1}))\subset \Delta(4n,1)$ for all $1\leq n\leq k$
\item $f_{k}^n|_{\overline{\Delta}(0,\frac{1}{2k+1})}$ is univalent for all $1\leq n\leq k$,
\item $f_{k}^n|_{U_k}$ is univalent for all $1\leq n\leq N_k+k$,
\item $f_k^j(x_n)=x_n^j$ for all  $1\leq j< N_n$, all $1\leq n\leq k$,
\item $f_k^{N_{n}}(x_n)=1$ for all $1\leq n\leq k$,
\item $f_k(1)=1$ and $f'_k(1)=\frac{1}{2}$.
\end{enumerate}
\end{lemma}


\begin{proof}[Proof of the Lemma:]
Observe that $\Omega \Subset U_0\subset  A_0$. We first construct a function $f_1$ that satisfies:
\begin{enumerate}
\item $f_{1}(U_1)\subset A_1$
\item $f_{1}(\overline\Delta(0,\frac{1}{3}))\subset \Delta(4,1)$
\item $f_{1}|_{\overline\Delta(0,\frac{1}{3})}$ is univalent,
\item $f_{1}^n|_{U_1}$ is univalent for all $1\leq n\leq 2$,
\item $f_1(x_1)=1$ 
\item $f_1(1)=1$ and $f'_1(1)=\frac{1}{2}$.
\end{enumerate}

Define compact sets $K_1=\{1\}$, $K_2=\{x_1\}$, $K_3= \overline{\Delta}(0,\frac{1}{3})$ and $K_4=U_1$ and observe that they are disjoint and their complement is connected. 
Let $h_1(z)=\frac{1}{2}z+\frac{1}{2}$, $h_2(z)=1$, $h_3(z)=z+4$ and let $h_4$ be a non-constant linear map that maps $h_4(K_1)\subset A_1$.  By Theorem \ref{approx}, for every $\epsilon_1 >0$ there exists an entire function $f_1$ that satisfies:
\begin{itemize}
\item $\|f_1-h_j\|_{K_j}\leq\epsilon_1$ for all $1\leq j\leq 4$
\item $f_1(x_1)=1$
\item $f_1(1)=1$ and $f'_1(1)=\frac{1}{2}$.
\end{itemize}  
Clearly (1) and (2) are satisfied as long as $\epsilon_1$ is sufficiently small. For (3) and (4) note that   $h_4$ is univalent on $U_1$  and $h_3$ is univalent on $\overline\Delta(0,\frac{1}{3})$ and that $h_4(U_1)\subset A_1\subset\overline\Delta(0,\frac{1}{3})$ , hence for $\epsilon_1$ sufficiently small the function $f_1$ is univalent on $U_1\cup\overline\Delta(0,\frac{1}{3})$ and $f_1^2$ is univalent on $U_1$. To conclude the first step we define  $x_2^j=f_1^j(x_2)$ for $1\leq j<3=N_2$.  Clearly $f_1$ satisfies all the conditions of the lemma for $k=1$.

\medskip

Now let us assume that we have already constructed entire functions $f_1,\ldots ,f_k$ and points $x_n^j$,  where $1\leq j<N_n$ and $1\leq n\leq k+1$,  that satisfy all the conditions $(a)-(h)$. Note that $N_{k+1}=N_k+k+1$. We proceed with the constructions satisfying the conditions for $k + 1$.
\medskip

Let us define compact sets  $ K_1:=\overline{\Delta}(0,4k-2)$, $ K_2:=\{x^{N_k+k}_{k+1}\}$.
By the induction hypothesis we have:
\begin{itemize}
\item $f_k^{N_k}(U_k)\subset \Delta(0,\frac{1}{2k+1})\backslash \overline{\Delta}(0,\frac{1}{2k+3})\subset \overline{\Delta}\left(0, \frac{1}{2k+1}\right)$,
\item $f^n_k|_{\overline{\Delta}\left(0, \frac{1}{2k+1}\right)}$ is univalent for all $1\leq n\leq k$,
\item $f^n_k\left(\overline{\Delta}\left(0, \frac{1}{2k+1}\right)\right)\subset \Delta(4n,1) $ for all $1\leq n\leq k$.
\end{itemize}   
This implies that compact sets $K_1$, $K_2$, $f_k^{k}\left(\overline{\Delta}\left(0, \frac{1}{2{k+3}}\right)\right)$ and $f^{N_k+k}_k(U_{k+1})$ are pairwise disjoint and have connected complement, hence there are compact sets $K_3$ and $K_4$ that satisfy:
\begin{itemize}
\item  $ f_k^{k}\left(\overline{\Delta}\left(0, \frac{1}{2{k+3}}\right)\right)\subset \text{int}K_3$,
\item  $ f^{N_k+k}_k(U_{k+1})\subset \text{int}K_4$,
\item $K_1$, $K_2$, $K_3$ and $K_4$ are pairwise disjoint and  have connected complement.
\end{itemize}

Now we define functions
 $$
 h_1(z):=f_k(z),\quad h_2(z):= 1, \quad h_3(z):= z+4,
 $$
and  let $h_4$ be a non-constant linear map that satisfies  $h_4(K_4)\subset A_{k+1}.$
 
 \medskip 
 
By Theorem \ref{approx}, for every $\epsilon_{k+1}>0$ there exists an entire function $f_{k+1}$, such that: 
\begin{enumerate}
\item $\|f_{k+1}-h_j\|_{K_j}\leq \epsilon_{k+1}$ for every $1\leq j\leq 4$, 
\item $f_{k+1}(x_n^j)=f_k(x_n^j)$ for all $0 \leq j< N_n$ and  $1\leq n \leq k$, (here $x_n^0:=x_n$)
\item $f_{k+1}(x_{k+1}^j)=f_k(x_{k+1}^j)$ for all $0 \leq j< N_k+k$, (here $x_{k+1}^0:=x_{k+1}$) 
\item $f_{k+1}(x^{N_k+k}_{k+1})=1$,
\item $f_{k+1}(1)=1$ and $f_{k+1}'(1)=\frac{1}{2}$,
\end{enumerate}

where we can choose  $\epsilon_{k+1}\leq \frac{1}{2^k}$ small enough such that:
\begin{enumerate}[label=(\roman*)]
\item $f_{k+1}^{N_{k+1}}(U_{k+1})\subset A_{k+1}$
\item $f_{k+1}^{n}(\overline{\Delta}(0,\frac{1}{2k+3}))\subset \Delta(4n,1)$ for all $1\leq n\leq k+1$
\item $f_{k+1}^n|_{\overline{\Delta}(0,\frac{1}{2k+3})}$ is univalent for all $1\leq n\leq k+1$,
\item $f_{k+1}^n|_{U_{k+1}}$ is univalent for all $1\leq n\leq N_{k+1}+k+1$,

\end{enumerate}

\begin{figure}[t]
\label{fig:oscillation}
\includegraphics[width=6in]{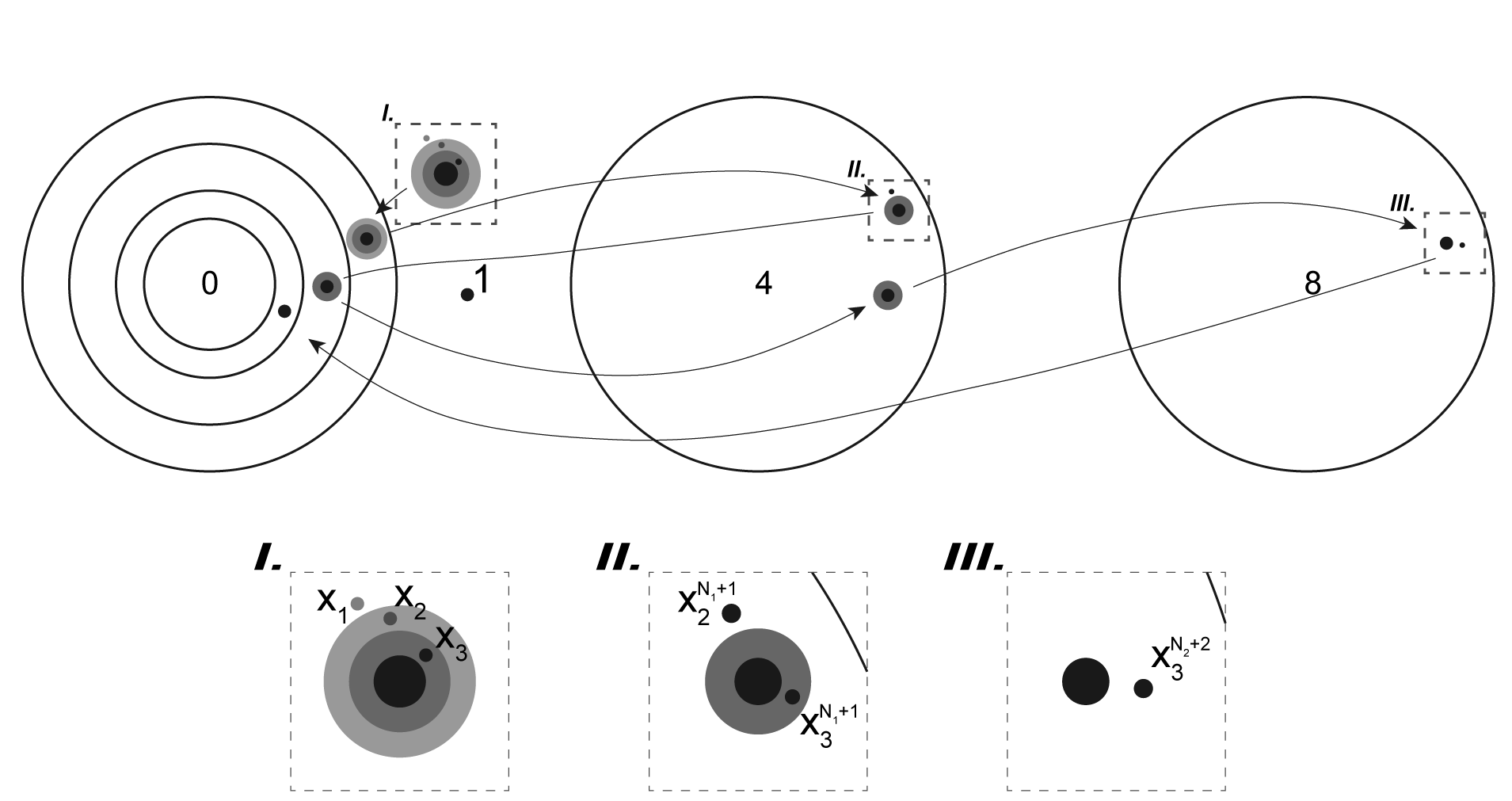}
\caption{The figure shows the first six iterates of $f$ acting on the wandering domain, which lies inside the black disk. There are three points around this disk which are part of the sequence of points that accumulates everywhere on the boundary of $\Omega$. Points $x_1$, $x_2^{N_1+1}$ and $x_3^{N_2+2}$, that belong to the orbit of points $x_1$, $x_2$ and $x_3$ respectively, and are mapped to the attracting fixed point $1$.}
\end{figure}

Indeed, for (i) recall that $N_{k+1}=N_k+ k+1$, hence for small $\epsilon_{k+1}$ we have 
$$f_{k+1}^{N_{k+1}}(U_{k+1})=f_{k+1}(f_{k+1}^{N_k+k}(U_{k+1}))\subset f_{k+1}(K_4)\subset A_{k+1}.$$

For (ii) and (iii) recall that by the induction hypothesis $f_{k}^n|_{\overline{\Delta}(0,\frac{1}{2k+1})}$ is univalent and $f_{k}^{n}(\overline{\Delta}(0,\frac{1}{2k+1}))\subset \Delta(4n,1)$  for all $1\leq n\leq k$. Since $f_{k}^{n}(\overline{\Delta}(0,\frac{1}{2k+1}))\subset K_1$ for $0\leq n< k$ and since (1) holds, we can choose $\epsilon_{k+1}$ so small that $f_{k+1}^{n}(\overline{\Delta}(0,\frac{1}{2k+3}))\subset \Delta(4n,1)$ and  that $f_{k+1}^n|_{\overline{\Delta}(0,\frac{1}{2k+3})}$ is univalent for all $1\leq n\leq k$. By taking even smaller $\epsilon_{k+1}$ if necessary we may assume that
$f_{k+1}^{k}(\overline{\Delta}(0,\frac{1}{2k+3}))\subset K_3$, and since  $h_3(z)=z+4$ is univalent on $K_3$ it follows that $f_{k+1}^{k+1}(\overline{\Delta}(0,\frac{1}{2k+3}))\subset \Delta(4(k+1),1)$ and that $f_{k+1}^{k+1}|_{\overline{\Delta}(0,\frac{1}{2k+3})}$ is univalent for all sufficiently small $\epsilon_{k+1}$.

For (iv) recall that by the induction hypothesis $f_{k}^n|_{U_{k}}$ is univalent for all $1\leq n\leq N_{k}+k$. Since $f_{k}^n(U_{k})\subset K_1$ for all  $0\leq n< N_{k}+k$ and since (1) holds, we can choose $\epsilon_{k+1}$ so small that $f_{k+1}^n|_{U_{k+1}}$ is univalent for all $1\leq n\leq N_{k}+k$. Next observe that by taking even smaller $\epsilon_{k+1}$ if necessary we can assume that $f_{k+1}^{N_{k}+k}(U_{k+1})\subset K_4$ and since $h_4$ is univalent on $K_4$ and $h_4(K_4)\subset A_{k+1}\subset\Delta(0,\frac{1}{2k+3})$, it follows that  $f_{k+1}^{N_{k+1}}|_{U_{k+1}}$ is univalent for sufficiently small $\epsilon_{k+1}$ (recall that $N_{k+1}=N_{k}+k+1$). Finally since $f_{k+1}^{N_{k+1}}(U_{k+1})\subset\Delta(0,\frac{1}{2k+3})$ and since (iii) holds we can conclude that $f_{k+1}^n|_{U_{k+1}}$ is univalent for all $1\leq n\leq N_{k+1}+k+1$.

\medskip

Finally we define $x_{k+2}^j:=f_{k+1}^j(x_{k+2})$ for all $1\leq j <N_{k+2}$. The entire function $f_{k+1}$ now satisfies all the conditions $(a)-(h)$ of the Lemma, hence this completes the inductive step. 
\end{proof}

Let us continue with the proof of Theorem 2. Let $N_n=\frac{n(n+1)}{2}$ and let $f$ be the limit of the sequence of entire functions $(f_n)$ given by the lemma above. Observe that for every $n\geq 1$ we have:

%

\begin{enumerate}
\item $f^{N_n}(\Omega )\subset \Delta(0,\frac{1}{2n+1})$, 
\item $f^{N_n+n}(\Omega)\subset \Delta(4n,1)$,
\item $f^n(\Omega)\subset \bigcup_{j\geq0 } \Delta(4j,1)$, 
\item $f^n|_{\Omega}$ is univalent,
\item $f^{N_{n}}(x_n)=1$, 
\item $f(1)=1$ and $f'(1)=\frac{1}{2}$.
\end{enumerate}

These properties clearly imply that $\Omega$ is contained in the Fatou set. Since the pre-images of the attracting fixed point $1$ accumulate everywhere on the boundary of $\Omega$ it follows that  $\Omega$ is a Fatou component. Finally properties $(1)$ and $(2)$ imply that $\Omega$ is an oscillating wandering domain.
\medskip

\begin{remark} Using the same argument as in Remark \ref{rem:1} we can conclude that the oscillating wandering domain from Theorem 2 is  \emph{eventually isometric}. Moreover it follows from the proof of Lemma \ref{lem:2} that the diameter of $f^n(\overline{\Omega})$ converges to $0$ as $n$ goes to infinity, hence  $\text{dist}_{\text{Eucl}}(f^n(z),\partial \Omega_n)\rightarrow 0$ for all $z\in \Omega$, e.g. all orbits of points converge towards the boundary. With a small adaptation of the proof, namely by choosing different linear map $h_3$ in every step of induction, we could also obtain a \emph{bungee} type behaviour in terms of the convergence of the orbits of points towards the boundary. 
\end{remark}

\section{Appendix}

Although the result of Theorem \ref{approx} is well known by many specialists in the theory of holomorphic approximation and could be attributed to several authors, e.g. \cite{BS49,F48,Roy67}, we were not able to find any reference where the result would be stated in such a form.  Note that a very similar approximation result was given by Eremenko-Lyubich, see \cite[Main Lemma]{EL87}.  For more information about the theory of holomorphic approximation we refer the interested reader to the recent survey \cite{FFW20}.

The proof of Theorem \ref{approx} that is presented bellow is a slight modification of the proof of the Main Lemma given Eremenko and Lyubich which relies on following lemma. 

\begin{lemma}\label{dense} Let $A$ be  a locally convex linear  topological space,  let $V$ be a domain  in $A$,  let $W$ be a convex  dense  subset in $V$ and let $S$ be an affine  subspace  of $A$  of  finite codimension,  such that $S\cap V \neq \emptyset$.   Then $S\cap W$ is dense  in  $S\cap V$.
\end{lemma}
For the proof of this lemma we refer the reader to \cite[p.420]{EL87}.

\begin{proof}[Proof of Theorem \ref{approx}]
Let $U_j$ be a simply connected open neighbourhood of $K_j$, such that  $f_j$ is holomorphic on $U_j$ and $U_j\cap K_i=\emptyset$ for all $j\neq i$. Consider the space $\mathcal{O}(U)$ of all holomorphic functions on $U:=\cup_{j=1}^nU_j$ equipped with the topology of uniform convergence on compact sets and let $F\in\mathcal{O}(U)$ be a function that satisfies $F|_{U_j}\equiv f_j$ for all $1\leq j\leq n $.  We define
$$
V:=\left\{g\in \mathcal{O}(U)\mid  \|g-f_j\|_{K_j}<\epsilon \text{ for all } 1\leq j\leq n   \right\}
$$
which is a convex set in $\mathcal{O}(U)$, and let be $W$ be the subset of $V$ consisting of polynomials. By Runge's theorem $W$ is dense in $V$. Moreover $W$ is also convex. Now consider an affine subspace of $\mathcal{O}(U)$ given by
$$
S:=\{g\in \mathcal{O}(U) \mid g(z)=f_j(z), g'(z)=f_j'(z) \text{ for all } z\in L_j   \text{ and all }    1\leq j\leq n\}.
$$
Since by the assumption all $L_j$ are finite sets it follows that $S$ is an affine subspace of $\mathcal{O}(U)$ of finite codimension. Finally since $F\in V\cap S$ it follows by Lemma \ref{dense} that $W\cap S$ is dense in $V\cap S$, hence there exists a polynomial $f$ such that for every $1\leq j\leq n$ we have
$$\|f-f_j\|_{K_j}<\epsilon, $$
and 
$$f(z)=f_j(z),\quad  f'(z)=f_j'(z) \quad \text{ for all } z\in L_j. $$

\end{proof}


\begin{thebibliography}{9999}
 \bibitem{Bar75} I. N. Baker, \emph{The domains of normality of an entire function.} Ann. Acad. Sci. Fenn. Math., 1 (1975), 277–283.

\bibitem{Bar76}I. N. Baker, \emph{An entire function which has wandering domains}. J. Austral. Math. Soc. Ser. A, 22(2):173-176, 1976

\bibitem{Bar84}I. N. Baker, \emph{Wandering domains in the iteration of entire functions.} Proc. London Math. Soc. (3), 49(3):563-576, 1984.

\bibitem{BF} A. M. Benini, N. Fagella, \emph{Singular values and bounded Siegel disks.} Math. Proc. Camb. Phil. Soc. (2018), 165, 249–265


\bibitem{BS49} H. Behnke, K. Stein, \emph{Entwicklung analytischer Funktionen auf Riemannschen Fl\"{a}chen.}Math. Ann. (120), 430--461, 1949.



\bibitem{BRS13} W. Bergweiler, P. J. Rippon, G. M. Stallard, \emph{Multiply connected wandering domains of entire functions.} Proc. Lond. Math. Soc. (3),107(6):1261-1301, 2013.

\bibitem{BEFRS19} A. M. Benini, V. Evdoridou, N. Fagella, P. Rippon, G. Stallard. \emph{Classifying simply connected wandering domains.} Preprint, arXiv:1910.04802, 2019

\bibitem{Bis15}C. J. Bishop. \emph{Constructing entire functions by quasiconformal folding}. Acta Math., 214(1):1-60, 2015.


\bibitem{BT21} L. Boc Thaler, \emph{Automorphisms of $\mathbb{C}^m$ with bounded wandering domains.}, Annali di Matematica (2021). https://doi.org/10.1007/s10231-020-01057-3

\bibitem{CR} A. Ch\'eritat, P. Roesch, \emph{Herman’s Condition and Siegel Disks of Bi-Critical
Polynomials.} Commun. Math. Phys. 344, 397–426 (2016)


\bibitem{EL87} A. E. Eremenko, M. Ju. Ljubich. \emph{Examples of entire functions with pathological dynamics.} J. London Math. Soc. (2), 36(3):458-468, 1987.

\bibitem{ERS20} V. Evdoridou, P. Rippon, G. Stallard, \emph{Oscillating simply connected wandering domains.} Preprint, arXiv:2011.14736, 2020


\bibitem{Fat20} P. Fatou,  \emph{Sur les  \'equations fonctionnelles.}, Bull. Soc. Math. France, 48:208–314, 1920.

\bibitem{FJL19} N.  Fagella,  X. Jarque, K.  Lazebnik, \emph{Univalent  wandering domains in the Eremenko-Lyubich class.} J. Anal. Math. 139, 369-395, 2019

\bibitem{FFW20}J. E. Forn\ae ss, F. Forstneri\v c, E. Wold, \emph{Holomorphic approximation:  the legacy of Weierstrass,  Runge,  Oka-Weil,  and  Mergelyan.}, Advancements in Complex Analysis, pp. 133–192. Springer, Cham, 2020

\bibitem{F48} H. Florack  \emph{Regul\"{a}re und meromorphe {F}unktionen auf nicht geschlossenen {R}iemannschen {F}l\"{a}chen.}, JSchr. Math. Inst. Univ. M\"{u}nster, 1, 1948


\bibitem{Her84}M. R. Herman, \emph{Exemples de fractions rationnelles ayant une orbite dense sur la sph\`ere de Riemann.} Bull. Soc. Math. France, 112(1):93-142, 1984.

\bibitem{KS} M. Kisaka and M. Shishikura, \emph{On multiply connected wandering domains of entire functions}, Transcendental dynamics and complex analysis, London Math. Soc. Lecture Note Ser. (348), Cambridge Univ. Press, Cambridge 2008, 217--250

\bibitem{MPS20}D. Marti-Pete, M. Shishikura. \emph{ Wandering domains for entire functions of finite order in the Eremenko–Lyubich class.} Proc. Lond. Math. Soc. (3),120(2):155-191, 2020.

\bibitem{Mil06}J. Milnor, \emph{Dynamics in one complex variable.} Third edition. Annals of Mathematics Studies, 160. Princeton University Press, Princeton, NJ, 2006


\bibitem{Roy67}H.L. Royden, \emph{Function theory on compact {R}iemann surfaces.} J. Analyse Math., 18:295–327, 1967.

\bibitem{Sch10}D. Schleicher, \emph{Dynamics of entire functions.} Holomorphic dynamical systems, 295–339, Lecture Notes in Math., 1998, Springer, Berlin, 2010.



\bibitem{Sul85}D. Sullivan. \emph{Quasiconformal homeomorphisms and dynamics. I. Solution of the Fatou-Julia problem on wandering domains.} Ann. of Math. (2), 122(3):401-418, 1985

\end{thebibliography}
\end{document}